\documentclass{ifacconf}

\usepackage{graphicx}      
\usepackage{natbib}        
\usepackage{amssymb}      
\usepackage{amsmath}
\usepackage{url}
\usepackage{color}
\usepackage{subcaption} 

\usepackage{float}
\usepackage{tcolorbox}
\tcbuselibrary{skins,breakable}
\usepackage{amssymb}  

\newtcolorbox{resp}[1][]{%
	enhanced jigsaw,%
	colback=gray!5!white,%
	colframe=gray!80!black,%
	size=small,%
	boxrule=1pt,%
	halign title=flush center,%
	coltitle=black,%
	breakable,%
	drop shadow=black!50!white,%
	attach boxed title to top left={xshift=1cm,yshift=-\tcboxedtitleheight/2,yshifttext=-\tcboxedtitleheight/2},%
	minipage boxed title=3cm,%
	boxed title style={%
		colback=white,%
		size=fbox,%
		boxrule=1pt,%
		boxsep=2pt,%
		underlay={%
			\coordinate (dotA) at ($(interior.west) + (-0.5pt,0)$);
			\coordinate (dotB) at ($(interior.east) + (0.5pt,0)$);
			\begin{scope}[gray!80!black]
				\fill (dotA) circle (2pt);
				\fill (dotB) circle (2pt);
			\end{scope}
		}%
	},%
	#1%
}

\allowdisplaybreaks

\newtheorem{theorem}{Theorem}
\newtheorem{lemma}{Lemma}
\newtheorem{remark}{Remark}
\newtheorem{definition}{Definition}
\newtheorem{assumption}{Assumption}
\newtheorem{corollary}{Corollary}
\newtheorem{problem}{Problem}
\newtheorem{example}{Example}

\def\nats{\mathbb{N}}

\def\AP{\mathit{AP}}

\newcommand{\LTLalways}{\Box}
\newcommand{\LTLeventually}{\Diamond}
\newcommand{\LTLuntil}{\mathbin{\sf U}}
\newcommand{\LTLnext}{\bigcirc}

\usepackage{color}

\begin{document}
\begin{frontmatter}

\title{Barrier Certificates for Uncertain Temporal Specifications} 

\thanks[footnoteinfo]{This work is supported by the following grants: EIC 101070802, and ERC 101089047.}

\author[First]{Mohammad H. Mamduhi}  
\author[First,Second]{Sadegh Soudjani} 

\address[First]{School of Computer Science, University of Birmingham, UK 
}
\address[Second]{Max Planck Institute for Software Systems, Germany 
}

\begin{abstract}                
This paper studies satisfying temporal logic specifications on stochastic dynamical systems, where the predicates evolve randomly over time. Such randomness may arise from uncertain environment models or external stochastic processes causing the sets associated with predicate satisfaction to vary in a non-deterministic manner. As a result, verifying whether a stochastic dynamical system satisfies a temporal specification depends also on the uncertainty in the predicates. We develop a certificate-based framework to bound the probability of satisfying temporal logic specifications with randomly evolving predicates. We first show that temporal logic specifications with stochastic predicates can be transformed to specifications with deterministic predicates on an augmented space which is extended to include the stochastic space of predicate's uncertainty. We then utilize barrier certificates on an augmented space to provide tractable optimization-based conditions and to avoid the computational burden of dynamic programming. Focusing on linear dynamics and safety-type specifications, we derive analytical conditions under which barrier certificates guarantee bounds on the probability of violating the stochastic safety predicates. The approach is demonstrated on numerical case studies.
\end{abstract}


\end{frontmatter}

\section{Introduction}

Ensuring safe operation of autonomous systems requires reasoning not only about their dynamics but also about complex task specifications that evolve over time. Temporal logic has emerged as a powerful framework for expressing such tasks, enabling formal descriptions of safety, sequencing, and reactive behaviors in a mathematically rigorous way \citep{Bellini2000}. Yet in real deployments, the truth values of the atomic predicates underlying these specifications are typically derived from perception modules that provide uncertain, probabilistic estimates of the environment \citep{sadigh2016safe}. An autonomous vehicle, for instance, rarely obtains an exact obstacle boundary or precise positions of surrounding agents; instead, it maintains stochastic representations of these quantities. Consequently, the atomic predicates used to evaluate a temporal logic formula may change randomly over time, and specification satisfaction becomes a probabilistic event rather than a deterministic one even for a given trajectory of the system.

Previous works on stochastic reachability problems (see, e.g. \citep{8882241, pmlr-v155-lew21a}) has shown that modeling environmental features as random sets provides a principled way to capture such uncertainties. These approaches offer tools for characterizing the probability of maintaining safety in the presence of stochastic disturbances or randomly varying safe sets. However, extending these ideas from single-step reach-avoid events to rich temporal logic specifications, and doing so in a scalable way, remains challenging. Dynamic programming–based methods can compute satisfaction probabilities, but their computational burden grows quickly with system dimension and specification  \citep{Soudjani2013adaptive,Survey2022automated}.

In this work, we study temporal logic specifications whose atomic predicates evolve according to stochastic processes, motivated by perception uncertainty in autonomous systems. Our objective is to bound the probability that a dynamical system satisfies such an uncertain specification. Direct computation is typically intractable due to the coupling between system trajectories and stochastic evolution of predicate sets.

To overcome this challenge, we introduce a certificate-based approach that provides computable and provable bounds on specification satisfaction probabilities. Specifically, we develop \emph{probabilistic barrier certificates} that allow us to reason about temporal logic satisfaction without enumerating all possible realizations of the random predicates or solving dynamic programming recursions. Our framework extends classical barrier-function ideas for safety verification to accommodate stochastic predicate dynamics and temporal operators in a unified way.

Focusing on linear systems and safety-type specifications, we derive analytic conditions under which barrier certificates yield bounds on the probability of violating stochastic safety predicates. These conditions lead to tractable computations and clarify how perception uncertainty propagates through the system to affect specification satisfaction. The resulting framework offers a computationally efficient alternative to dynamic-programming-based techniques while enabling the analysis of specifications that go beyond traditional reach-avoid formulations.

\noindent\textbf{Related Works.} Modeling task requirements with temporal logic specifications is a common method for the synthesis of safety-critical control systems \citep{4177895}. Control and planning of uncertain systems under temporal logic specifications have been explored mostly with external disturbances affecting the system dynamics \citep{Pian_etal_2024,schon2024bayesian}. Perception-based temporal logic planning problem is investigated for uncertain semantics maps that are continuously learned by semantic SLAM methods \citep{9708724}. Risk-based temporal logic specifications is studied in \cite{9103559,engelaar2024risk} by transforming stochastic risks to risk-tightened deterministic temporal logic constraints.

Barrier certificates have been used for ensuring satisfaction of temporal specifications for stochastic systems \citep{9157966}. Majority of the works focus on uncertain system dynamics \citep{SANTOYO2021109439}. Safety and reachability of stochastic hybrid systems have been addressed under random specifications using dynamic programming for controller synthesis \citep{SUMMERS20132906}. In \cite{9800957} barrier certificates are used to ensure probabilistic safety when system dynamics and barrier function are unknown and estimated from data by Gaussian processes or worst-case uncertainty in both the system and barrier function. Stochastic control barrier functions have been introduced to derive less conservative safety probability bounds for stochastic systems under deterministic safety specifications \citep{9683349}. 

To the best of our knowledge, none of these works apply barrier certificates on the case where temporal logic predicates evolve randomly and the specification is intrinsically stochastic. This work integrates stochastic perception-based task specification models, barrier-certificate methods, and temporal logic reasoning to provide a principled and scalable approach for quantifying safety guarantees under uncertain, randomly evolving specifications.

\textbf{Notations.} In this paper, $\{x_k\}_{k=0}^t$ denotes the sequence $x_0,\ldots, x_t$. The sets of natural, real, and non-negative real numbers are denoted by $\mathbb{N}$, $\mathbb{R}$, and $\mathbb{R}_0^+$, respectively, and $\mathbb N_0=\mathbb N\cup \{0\}$. A random variable $v\sim \mathcal D(\mu_v,\sigma_v^2)$ has the distribution $\mathcal D$ with the mean and variance of $\mu_v$ and $\sigma_v^2$. By $\oplus$, $\emptyset$, $\textsf{Tr}(\cdot)$, $(\cdot)^\top$, and $\textsf{E}_v(\cdot)$ we denote, respectively, the Minkowski summation operator, empty set, trace operator, transpose, and expectation with respect to a random variable $v$. We use the shorthand notation $(x,y)=[x^\top,y^\top]^\top$. For a vector $y\in \mathbb{R}^n$, and matrix $P\in \mathbb{R}^{n\times n}$, $\|y\|_P^2=y^\top Py$. For a square matrix $A$, $A\succeq 0$ ($\succ0$) denotes $A$ is positive semi-definite (positive definite), and $A\preceq 0$ ($\prec 0$) denotes $A$ is negative semi-definite (negative definite). For two sample spaces $S_1$ and $S_2$, $S_1\times S_2$ is the product space. For two sets $D_1$ and $D_2$, $D_1\cap D_2$ denotes their intersection, and $D_1\backslash D_2$ denotes $D_1$ excluding $D_2$.

\section{Model \& Problem Statement}
\noindent
\textbf{Dynamical System.}
We consider a discrete time dynamical system denoted as $\Lambda$ in a stochastic state space $X\subseteq \mathbb{R}^n$ associated with the Borel $\sigma$-algebra $\mathcal{B}(X)$, in the following form:
\begin{equation}\label{sys_model}
    \Lambda: x_{k+1}=f(x_k,u_k,w_k), \;\; k\in\mathbb{N}_0,
\end{equation}
where $x_k\in X$, $u_k\in \mathcal{U}$, and $w_k\sim \mathcal{D}(0,\sigma_w^2)$ are, respectively, the system's state, control input, and stochastic uncertainty with distribution $\mathcal D$. The control input space is denoted by a compact Borel set $\mathcal{U}\subseteq \mathbb{R}^m$. Then for each state $x\in X$ and control input $u\in \mathcal U$, we can assign a probability measure $ T (\cdot|x, u)$, where $ T:\mathcal B(X)\times X\times \mathcal U \rightarrow [0,1]$ is the controlled transition probability function. 

\noindent
\textbf{Linear Temporal Logic.}
We consider specifications in Linear Temporal Logic (LTL) \citep{baier2008principles}. 
Formulas in LTL are constructed inductively over a set of atomic propositions $\AP$ according to the syntax
\begin{equation*}
	\psi :=  p \mid \neg \psi \mid \psi_1 \wedge \psi_2 \mid \mathord{\bigcirc} \psi \mid \psi_1\LTLuntil \psi_2,
\end{equation*}
where $p\in\AP$. 
The semantics of LTL is defined on infinite sequences of elements from $2^{\AP}$.
Let $\rho=\rho_0,\rho_1,\dots$ be an infinite sequence of elements from $2^{\AP}$ and define $\rho[i]=\rho_i,\rho_{i+1},\dots$ for any $i\in \nats_0$. Then  
the satisfaction relation between $\rho$ and a property $\psi$, expressed in LTL, is denoted by $\rho\models\psi$. 
We have $\rho\models p$ if $p\in \rho_0$.
Furthermore,
$\rho\models \neg \psi$  if $\rho\not\models\psi$ and 
$\rho\models \psi_1\wedge\psi_2$ 
if $ \rho\models \psi_1$ and $\rho\models \psi_2$.
For \emph{next} operator, $\rho\models\mathord{\bigcirc}\psi$ holds if $\rho[1]\models\psi$.
The \emph{until} operator $\rho\models \psi_1\LTLuntil\psi_2$  holds if $ \exists i \in \mathbb N_0:$ $\rho[i] \models \psi_2, \mbox{and } 
\forall j \in \mathbb N_0, j<i, \rho[j]\models \psi_1
$.
We define derived operators such as disjunction ($\vee$), eventually ($\LTLeventually  $), and globally ($\LTLalways$) in the usual way.

\noindent
\textbf{Uncertain Labeling Functions.}
Consider a labeling function $L:X\times\Theta\rightarrow 2^\AP$ that maps any pair of states and parameters $(x,\theta)$ to a subset of atomic predicates, where $x\in X$ and $\theta\in \Theta$. For any sequence of parameters $\theta = \theta_0,\theta_1,\theta_2,\ldots$, the mapping $L$ maps any sequence of states $\xi = x_0,x_1,x_2,\ldots$ to a sequence $\rho =\rho_0,\rho_1,\rho_2,\ldots$ with $\rho_i = L(x_i,\theta_i)$. Using this labeling function, we can define satisfaction relation between a trajectory $\xi$ and a specification $\psi$, denoted as $(\xi,\theta)\models \psi$.

The traditional definition of labeling function for LTL satisfaction is defined directly on the state space $X$ and does not include the parameter space $\Theta$. In this work, we aim to design control policies to satisfy specifications under uncertainty in the labeling function $L$ through $\theta$.

\begin{resp}
\begin{problem}
\label{prob_prob}
Suppose $\theta\in\Theta$ is randomly changing over time. Design a control policy $\{u_k\}_{k=0}^T$ for $\Lambda$ such that the probability of satisfying a given specification $\psi$ is at least a given threshold  $\varepsilon$, i.e.,
\begin{equation}
    \mathsf{P}_{w,\theta}\{(\xi,\theta)\models \psi\}\geq \varepsilon.
\end{equation}
\end{problem}
\end{resp}

\medskip
\begin{example}[Safety Specifications]
Consider an uncertain safety specification with uncertain initial and unsafe sets as follows. Take the uncertain parameter $\theta = (\theta_{\textsf{i}},\theta_{\textsf{u}})\in \Theta_{\textsf{i}}\times \Theta_{\textsf{u}}$, where $\theta_{\textsf{i}}$ affects the initial set, and $\theta_{\textsf{u}}$ affects the unsafe set.
Define $X_{\textsf{i}}(\cdot)\!:\! \Theta_{\textsf{i}} \!\rightarrow\! \mathcal{B}(X)$ as the uncertain initial set, and $X_{\textsf{u}}(\cdot)\!:\! \Theta_{\textsf{u}} \!\rightarrow \!\mathcal{B}(X)$ as the uncertain unsafe set. 
We can represent 
the stochastic safety specification $\Upsilon=(\Theta_{\textsf{i}},\Theta_{\textsf{u}},X_{\textsf{i}}, X_{\textsf{u}}, T)$, where $T$ is the safety specification horizon. The aim is to find the safety probability for $\Lambda$ under synthesized control inputs $\{u_k\}_{k\in \mathbb{N}}$ that satisfies $\Upsilon$, i.e., starting from any initial state $x_0\in X_{\textsf{i}}(\theta_{\textsf{i}})$, system trajectories should avoid entering the unsafe set $X_{\textsf{u}}(\theta_{\textsf{u}})$ over the entire horizon $T$ with a probability $\varepsilon\in (0,1)$.

The uncertain safety specification fits the above definition of LTL satisfaction as follows. Define the labeling function $L:X\times\Theta_{\textsf{i}}\times\Theta_{\textsf{u}}\rightarrow 2^{AP}$ with $AP = \{\textsf{Init},\textsf{Unsafe}\}$ such that $L(x,\theta_{\textsf{i}},\theta_{\textsf{u}}) = \textsf{Init}$ iff $x\in X_{\textsf{i}}(\theta_{\textsf{i}})$, and $L(x,\theta_{\textsf{i}},\theta_{\textsf{u}}) = \textsf{Unsafe}$ iff $x\in X_{\textsf{u}}(\theta_{\textsf{u}})$.
Also define the safety specification in LTL notation as $\Upsilon = \textsf{Init} \wedge\left(\wedge_{n=0}^T\LTLnext^n \neg\textsf{Unsafe}\right)$. 
\end{example}

In the next section, we show how to tackle Problem~\ref{prob_prob} by assuming probability kernels for the random parameters.


\section{Solution Approach}
\label{sec:solution}

Here we express the dynamics of $\Lambda$ in \eqref{sys_model} transformed on the extended state space including the stochastic space of random set parameter $\theta$. The main advantage of this transformation, as we discuss in detail in the next section, is that the challenging problem of designing safety controllers for $\Lambda$ with stochastic set parameter can be solved equivalently by taking expectation of stochastic events on deterministic sets defined on the extended state space. 

\textbf{Modeling Random Parameters.}
Solving Problem~\ref{prob_prob} requires formalizing the probabilistic behavior of the random parameters.
This section provides the details of a general and realistic model for uncertainty and shows how to solve the problem by constructing an augmented system with an extended state space.

\begin{assumption}
Consider the sequence of uncertain parameters $\theta = \theta_0,\theta_1,\ldots$, where the initial parameter $\theta_0$ follows a given distribution $\mathcal D_0$, and the subsequent parameters $\theta_{k+1}$ are generated randomly according to a conditional probability kernel $T_\theta: \Theta\times\mathcal{B}(\Theta)\rightarrow [0,1]$, where $\theta_{k+1}\sim T_\theta(\theta_k,\cdot)$ for all $k\in\mathbb N_0$. \qed
\end{assumption}

\begin{remark}
In perception-based setting, the probability kernel $T_{\theta}$, which determines the probabilistic evolution of $\theta$, can be obtained by parameterizing or learning the accuracy and reliability of the perceptive sensors and their correlations with the geometric distance between the sensors (e.g., camera) and the object (e.g., a wall).\qed
\end{remark}

\textbf{Augmented System.}
We use the following lemma to equivalently model the parameter sequence as a dynamical system and build an augmented system.

\begin{lemma}\label{lemma1}\cite[Lemma 2.22/3.22]{Bak:Kallenberg}
Let a discrete time Markov process $\{\pi_k\}_{k\in\mathbb{N}}$ with transition kernel $T(\varpi,A)=\textsf{P}[\pi_{k+1}\in A| \pi_k=\varpi]$ be defined on a Borel space. Then there exists a measurable function $h:Y\times [0,1]\rightarrow Y$, and the i.i.d. uniformly distributed process $\{v_k\}_{k\geq 1}$ on $[0,1]$ such that $\pi_{k+1}=h(\pi_k,v_{k+1}), \forall k\geq 0$ produces a process with the same same distribution, and transition kernel $T$.\qed
\end{lemma}

According to Lemma \ref{lemma1}, we can express the dynamics of the random state $\theta$ as
\begin{equation}
    \theta_{k+1}=h(\theta_k, v_{k+1}), \;\;v_k\sim U(0,1), \;\forall k.
\end{equation}

For any state $x\in X$, we define the discrete time augmented state $z\triangleq(x,\theta)$ on the product space $\Xi=X\times \Theta$. Denoting the augmented system with the augmented state $z$ by $\bar\Lambda$ which evolves on the product space $\Xi$, we have
\begin{equation}
    \bar\Lambda: z_{k+1}=(x_{k+1},\theta_{k+1})=\begin{cases}f(x_k,u_k,w_k),\\h(\theta_k, v_{k+1}),\end{cases} \; z_k\in \Xi.
\end{equation}
The augmented system is hence defined by the augmented tuple $\bar\Lambda=(\Xi, \mathcal U, f,h, w,v)$.

The stochastic transition kernel for the augmented state $z$ is a time-varying probability transition $\bar{T}: \mathcal B(X)\times \Xi \times \mathcal U\rightarrow [0,1]$ such that
\begin{equation}\label{trans_ker}
    \bar{T}_k(z_{k+1}|z_k, u_k)=T (x_{k+1}|x_k,u_k)T_{\theta}(\theta_{k+1}|\theta_k), \; \forall k.
\end{equation}
The transition probabilities are controlled by the control input $u\in \mathcal U$. Let a control policy $\{u_k\}_{k\in \mathbb{N}_0}$ be a sequence of measurable maps $u_k: \Xi \rightarrow \mathcal{U}$. Under this policy and starting from the initial state $z_0=(x_0, \theta_0)\in \Xi$, the evolution of the augmented system $\bar\Lambda$ over the time horizon $[0,T]$ is fully captured by the stochastic process $\{z_k=(x_k,\theta_k)\}_{k=0}^T$ on the sample space $\Omega=\Xi^{T+1}$ with the associated $\sigma$-algebra $\mathcal B(\Omega)$.



\textbf{Solution Approach.} So far, we transformed the system evolution from the state space $X$ to the product space $\Xi$ by augmenting the state vector and extending the state space. 
Depending on the type of the specification, techniques from the literature can be used to solve Problem~\ref{prob_prob} on the augmented system $\bar\Lambda=(\Xi, \mathcal U, f,h, w,v)$. The available approaches utilize dynamic programming characterizations for safety and reachability specifications \citep{SUMMERS20132906,Soudjani2013adaptive}, abstraction to interval Markov decision processes \citep{zhang2024formal} for handling all LTL specifications, and certificate-based techniques for safety and other temporal specifications \citep{9157966,Abate_et.al.}. 

While all the aforementioned techniques are generally applicable to the augmented system $\bar\Lambda$, we will focus on barrier certificates and safety specifications as a scalable discretization-free approach. The main difficulty in applying barrier-based approaches to $\bar\Lambda$ is to handle the random initial parameter. While barrier-based approaches assume a given set of initial states and provide a lower bound on satisfaction probability that holds for all such initial states, this is not directly applicable to the augmented system since the initial parameter follows a given distribution $\theta_0\sim\mathcal D_0$. To handle this challenge, the barrier certificate and the lower bound on the satisfaction probability can be computed as a function of the initial parameter $\theta_0$. The expectation of such a lower bound with respect to $\mathcal D_0$ will give a solution to Problem~\ref{prob_prob}.  

To formalize the details, the goal is to find a control input sequence $\{u_k\}_{k=0}^T$ that guarantees probabilistic safety for the augmented system $\bar\Lambda$ over the horizon $[0,T]$ on the sample space $\Omega$. In fact, we aim to synthesize $\{u_k\}_{k=0}^T$ such that the evolution of the controlled state $\{x_k\}_{k=0}^T$ starting from the initial set $X_{\textsf{i}}(\theta_{0})$, and under the control policy $u$, is probabilistically guaranteed to remain in $X\backslash X_{\textsf{u}}(\theta)$ over the horizon $[0,T]$, i.e., for $\varepsilon \in(0,1]$, we look for
\begin{equation}\label{safe_prob1}
   \!\! p_{safe}=\!\textsf{P}_{\!u, z_0}\{x_k\!\notin\! X_{\textsf{u}}(\theta), \forall k\!\in\! [0,T]\big| x_0\!\in\!  X_{\textsf{i}}(\theta_0)\}\geq \varepsilon.
\end{equation}
By defining two deterministic sets $\bar X_{\textsf{u}}=\{(x,\theta)\in \Xi \;|\; x\in X_{\textsf{u}}(\theta)\}$ and $\bar X_{\textsf{i}}=\{(x,\theta)\in \Xi \;|\; x_0\in X_{\textsf{i}}(\theta_0)\}$ on the product space $\Xi$, the safety probability with random sets becomes equivalent to the safety probability with deterministic sets on the augmented system.
We can, therefore, solve the equivalent deterministic problem on the extended state space, with the augmented system $\bar\Lambda=(\Xi, \mathcal U, f, h, w, v)$, and under the deterministic specification $\bar\Upsilon=(\bar X_{\textsf{i}}, \bar X_{\textsf{u}}, T)$. 

In the next section, we provide the analytical derivations for linear systems and specific models of the random sets.



\section{Linear Systems with Uncertain Specifications}
\label{sec:ext_space}


\textbf{Uncertain Temporal Specification Model.} Let $X_{\textsf{i}}^{\text{no}}$ and $X_{\textsf{u}}^{\text{no}}$ be two convex regions with shape kernels $R_{\textsf{i}}$ and $R_{\textsf{u}}$, respectively, and centered at $\bar c_{\textsf{i}}$ and $\bar c_{\textsf{u}}$. We can express the deterministic nominal initial and unsafe sets as
\begin{equation*}
   X_{\textsf{i}}^{\text{no}}= \bar{c}_{\textsf{i}} \oplus \bar{s}_{\textsf{i}}R_{\textsf{i}},\;\;\;\;\; X_{\textsf{u}}^{\text{no}}= \bar{c}_{\textsf{u}} \oplus \bar{s}_{\textsf{u}}R_{\textsf{u}},
\end{equation*}
where $\bar{s}_{\textsf{i}}$ and $\bar{s}_{\textsf{u}}$ are the sizes of shape kernels $R_{\textsf{i}}$ and $R_{\textsf{u}}$. We further assume that $X_{\textsf{i}}^{\text{no}}\cap X_{\textsf{u}}^{\text{no}}=\emptyset$.

Consider a random variable $\theta_{\textsf{i}}\sim\mathcal{D}_{\textsf{i}}(\mu_{\textsf{i}}, \sigma_{\textsf{i}}^2)$ defined on a Borel space $\Theta_{\textsf{i}}$. Moreover, let $\theta_{\textsf{u}}=\{\theta_{\textsf{u},k}\}_{k\in\mathbb{N}}$ be a random process defined on a Borel space $\Theta_{\textsf{u}}$, where for each $k$, $\theta_{\textsf{u},k}$ is a random variable selected from a known distribution $\mathcal{D}_{\textsf{u}}(\mu_{\textsf{u}}, \sigma_{\textsf{u}}^2)$, and the elements in $ \{\theta_{\textsf{u},k}\}_{k\in\mathbb{N}}$ are mutually independent and identically distributed (i.i.d.). We also assume that $\theta_{\textsf{i}}$ is independent of the random elements in the process $\{\theta_{\textsf{u},k}\}_{k=0}^T$. 

We define Borel-measurable kernels $T_{\textsf{i}}: \mathcal{B}(\Theta_{\textsf{i}})\times \Theta_{\textsf{i}}\rightarrow [0,1]$ and $T_{\textsf{u}}:\mathcal{B}(\Theta_{\textsf{u}})\times \Theta_{\textsf{u}}\rightarrow [0,1]$ as probability measures\footnote{Since $\{\theta_{\textsf{u},k}\}_{k=0}^T$ are i.i.d., they have identical and unconditional probability transition kernels.}, respectively, for $\theta_{\textsf{i}}\in \Theta_{\textsf{i}}$ and $\theta_{\textsf{u},k}\in \Theta_{\textsf{u}}$ on their corresponding Borel spaces $(\Theta_{\textsf{i}},\mathcal{B}(\Theta_{\textsf{i}}))$ and $(\Theta_{\textsf{u}},\mathcal{B}(\Theta_{\textsf{u}}))$.

Define $\theta\triangleq\{\theta_{\textsf{i}},\theta_{\textsf{u}} \}\in \Theta=\Theta_{\textsf{i}}\times \Theta_{\textsf{u}}$, with the transition kernel $T_{\theta}(\cdot)=T_{\textsf{i}}(\cdot)T_{\textsf{u}}(\cdot)$. Since $\theta$ includes i.i.d. random variables, the transition kernel $T_\theta$ becomes unconditional, i.e., independent of its past realizations\footnote{Note that $\theta_{\textsf{i}}$ is a single independent random variable, while $\theta_{\textsf{u}}$ is a random process with i.i.d. states $\theta_{\textsf{u},k}$, for each $k$.}. The uncertain initial and unsafe sets can then be expressed as
    \begin{subequations}
    \begin{align}\label{X_0_theta}
    X_{\textsf{i}}(\theta_{\textsf{i}})&= \bar{c}_{\textsf{i}} \oplus (\bar{s}_{\textsf{i}}+\theta_{\textsf{i}})R_{\textsf{i}}, \\\label{X_u_theta}
    X_{\textsf{u}}(\theta_{\textsf{u}})&= \bar{c}_{\textsf{u}} \oplus (\bar{s}_{\textsf{u}}+\theta_{\textsf{u}})R_{\textsf{u}}.
\end{align}
  \end{subequations}
Temporal safety specification is then described as starting from the initial set $X_{\textsf{i}}(\theta_{\textsf{i}})$, avoid entering the unsafe set $X_{\textsf{u}}(\theta_{\textsf{u}})$. Equivalently, we can write the safety specification in the short tuple-format as $\Upsilon =(X_{\textsf{i}}(\theta_{\textsf{i}}),X_{\textsf{u}}(\theta_{\textsf{u}}), T)$, or in the LTL formula by $\Upsilon = \textsf{Init} \wedge\left(\wedge_{n=0}^T\LTLnext^n \neg\textsf{Unsafe}\right)$.

%


Next, we construct our augmented stochastic linear time-invariant (LTI) system on the extended state space, and then we define a modified control barrier function (CBF) to probabilistically guarantee the separation of augmented system trajectories from the unsafe set. 

\textbf{CBF for Augmented Stochastic LTI Systems.} Since our focus in this work is on modifying CBF with respect to augmented systems on extended stochastic state spaces, we restrict our attention on discrete-time stochastic linear time-invariant (SLTI) systems of the following form
\begin{equation}\label{sys_model_lin}
    \Lambda_{\text{SLTI}}:x_{k+1}=Ax_k+Bu_k+w_k,
\end{equation}
where $x_k\in \mathbb R^n$, $u_k\in \mathbb R^m$, $A$ and $B$ are constant matrices of appropriate dimensions, and $w_k\sim\mathcal N(0,\sigma_w^2)$ is an i.i.d. random process. We fix the control policy to be linear state feedback, i.e., $u_k=Lx_k\in\mathcal U, k\in[0,T-1]$, where $L$ is the control gain to be designed such that $\Lambda_{\text{SLTI}}$ is safe with respect to the specification $\Upsilon$. We first express our augmented system $\bar\Lambda_{\text{SLTI}}$, with the augmented state vector $z_k\triangleq(x_k,\theta_k)$, where $\theta_k=(\theta_{\textsf{i}}, \theta_{\textsf{u},k})$. Recalling that $\theta_{\textsf{i}}$ is a random variable that is selected once at time $k=0$ (not evolving with time), and that $\theta_{\textsf{u}}$ is a sequence of i.i.d. random variables $\{\omega_k\}_{k\in\mathbb N}$ (see \textit{Lemma 1}), we derive the evolution model for the set random parameter $\theta$
\begin{align*}
    \theta_{k+1}=\begin{pmatrix}
        \theta_{\textsf{i}}\\\theta_{\textsf{u},k+1}\end{pmatrix}=\begin{pmatrix}
        \theta_{\textsf{i}}\\\omega_{k+1}
    \end{pmatrix}\triangleq \bar{\omega}_k,
\end{align*}
where the last expression above is a valid definition due to the time homogeneity of the random sequence $\{\omega_k\}_{k\in\mathbb N}$. Now for the augmented state $z_k$ we have
\begin{equation*}
  \bar\Lambda_{\text{SLTI}}\!:\!  z_{k+1}=\underbrace{\!\begin{pmatrix} \!A_{cl} &\mathbf{0}_{n\times 2n} \!\\ \!\mathbf{0}_{2n\times n}&\mathbf{0}_{2n\times 2n}\!\end{pmatrix}\!}_{\bar A}z_k+\underbrace{\begin{pmatrix}I_{n} & \mathbf{0}_{n\times 2n}\\\mathbf{0}_{2n\times n}& I_{2n}
  \end{pmatrix}}_{\bar D}\!\!\begin{pmatrix} w_k\\ \bar{\omega}_k\end{pmatrix}\!,
\end{equation*}
with $A_{cl}=A+BL$. Therefore, we can the write the augmented model in a compact form as
\begin{equation}\label{aug_sys_LTI2}
     \bar\Lambda_{\text{SLTI}}:  z_{k+1}=\bar{A}z_k+\bar{D}\mathsf{w}_k, \;\;\mathsf{w}_k=(w_k, \bar{\omega}_k).
\end{equation}
Now define the deterministic initial and unsafe sets as
\begin{subequations}
    \begin{align}\label{X_0}
        \bar{X}_{\textsf{i}}&=\{z=(x,\theta)\in \Xi\times \mathcal{U} \;|\; x\in X_{\textsf{i}}(\theta_{\textsf{i}})\},\\\label{X_u}
        \bar{X}_{\textsf{u}}&=\{z=(x,\theta)\in \Xi\times \mathcal{U} \;|\; x\in X_{\textsf{u}}(\theta_{\textsf{u}})\},
    \end{align}
\end{subequations}
and define the deterministic temporal specification on the augmented space as $\bar\Upsilon=(\Theta_{\textsf{i}}, \Theta_{\textsf{u}}, \bar X_{\textsf{i}},  \bar X_{\textsf{u}}, T)$.

Before constructing the CBF for the augmented system $\bar\Lambda_{\text{SLTI}}$, we need to address the probability of overlap between the initial and the unsafe sets $X_{\textsf{i}}(\theta_{\textsf{i}})$ and $X_{\textsf{u}}(\theta_{\textsf{u}})$.

\begin{remark}
Due to unsupported ranges of stochastic variables $\theta_{\textsf{i}}$ and $\theta_{\textsf{u},k}$, there is a non-zero probability that $X_{\textsf{i}}(\theta_{\textsf{i}}) \cap X_{\textsf{u}}(\theta_{\textsf{u},0})\neq  \emptyset$, which violates the safety condition for the initial state. Hence no barrier level sets can be found to separate the unsafe set from the initial set. Since in our setting, overlapping event occurs with non-zero probability, we restrict the definition of the CBF given the condition $X_{\textsf{i}}(\theta_{\textsf{i}}) \cap X_{\textsf{u}}(\theta_{\textsf{u},0})=\emptyset$ holds. \qed
\end{remark}

We analytically derive $p_{\emptyset}=\textsf{P}_{\theta_{\textsf{i}},\theta_{\textsf{u}}}[X_{\textsf{i}}(\theta_{\textsf{i}}) \cap X_{\textsf{u}}(\theta_{\textsf{u},0})=  \emptyset]$, which is the probability that safety specifications are separable at $k=0$ by any barrier function. For the ease of presentation and derivation, in the rest of this paper we assume $R_{\textsf{i}}=R_{\textsf{u}}$, i.e., the initial and unsafe sets are of identical shape. The results of this work extend to the case $R_{\textsf{i}}\neq R_{\textsf{u}}$ as long as $R_{\textsf{i}}$ and $R_{\textsf{u}}$ remain convex.

\begin{lemma}
Consider the uncertain initial and unsafe sets as defined in \eqref{X_0_theta} and \eqref{X_u_theta}, with $R_{\textsf{i}}=R_{\textsf{u}}$, $\theta_{\textsf{u}}=\{\theta_{\textsf{u},k}\}_{k\in[0,T]}\sim\mathcal N(0,\sigma_{\textsf{u}}^2), \forall k$, and $\theta_{\textsf{i}}\sim\mathcal N(0,\sigma_{\textsf{i}}^2)$. Let $F_{\theta_{\textsf{i}}}$ and $F_{\theta_{\textsf{u}}}$ be the cumulative distribution functions (CDF) of $\theta_{\textsf{i}}$ and the $k^{\textsf{th}}$ element\footnote{Since elements of $\theta_{\textsf{u}}$ are i.i.d., their CDFs are identical, therefor, we denote $\text{CDF}(\theta_{\textsf{u},k})=F_{\theta_{\textsf{u}}}, \forall k$.} of $\theta_{\textsf{u}}$ (i.e., $\theta_{\textsf{u},k}$), respectively. Then the probability of overlap between $X_{\textsf{u}}(\theta_{\textsf{u}})$ and $X_{\textsf{i}}(\theta_{\textsf{i}})$ at any time $k$ including $k=0$, denoted by $p_k^{ol}$, is 
\begin{align*}
    p_k^{ol}&=\textsf{P}_{\theta_{\textsf{i}},\theta_{\textsf{u}}}[X_{\textsf{u}}(\theta_{\textsf{u},k})\cap X_{\textsf{i}}(\theta_{\textsf{i}})\neq \emptyset]\\&=1-F_{\theta_{\textsf{i}}+\theta_{\textsf{u}}}(d_R-(\bar s_{\textsf{u}}+\bar s_{\textsf{i}})),
\end{align*}
where $F_{\theta_{\textsf{i}}+\theta_{\textsf{u}}}$ is the CDF of $\theta_{\textsf{i}}+\theta_{\textsf{u},k}$, and $d_R=\|\bar c_{\textsf{u}}-\bar c_{\textsf{i}}\|_R=\inf\{\lambda\geq 0: \bar c_{\textsf{u}}-\bar c_{\textsf{i}}\in \lambda R\}$ is the $R$-norm measured distance between the centers of the sets.                          
\end{lemma}

\textit{Proof.}
The condition $X_{\textsf{u}}(\theta_{\textsf{u},k})\cap X_{\textsf{i}}(\theta_{\textsf{i}})\neq \emptyset$ is equivalent to $d_R\leq \bar s_{\textsf{u}}+\theta_{\textsf{u},k}+ \bar s_{\textsf{i}}+\theta_{\textsf{i}}$, or with tweaking the parameters, is equivalent to $\theta_{\textsf{u},k}+ \theta_{\textsf{i}}\geq d_R-(\bar s_{\textsf{u}}+\bar s_{\textsf{i}})$. Define $\bar d\triangleq d_R-(\bar s_{\textsf{u}}+\bar s_{\textsf{i}})$. Since $\theta_{\textsf{i}}$ and $\theta_{\textsf{u},k}$ are independent $\forall k$, we have
\begin{align*}
    p_{\emptyset}=\mathsf{P}[\theta_{\textsf{i}}+\theta_{\textsf{u},k}\leq \bar d]=F_{\theta_{\textsf{i}}+\theta_{\textsf{u}}}(\bar d).
\end{align*}
Hence the overlap probability $p^{ol}_k=\mathsf{P}[\theta_{\textsf{i}}+\theta_{\textsf{u},k}\geq \bar d]=1-F_{\theta_{\textsf{i}}+\theta_{\textsf{u}}}(\bar d)$. Since $\{\theta_{\textsf{u},k}\}_{k\in[0,T]}$ are i.i.d. $\forall k$, and they are independent of $\theta_{\textsf{i}}$, and the fact that $\text{CDF}(\theta_{\textsf{u},k})=F_{\theta_{\textsf{u}}}, \forall k$, the overlap event at each time is independent of the previous  events, and overlap probabilities are equal $\forall k[0,T]$, and the proof is complete. \qed
    
We now define the modified CBF conditions for the augmented system $\bar\Lambda_{\text{SLTI}}$, with augmented state $z_k=(x_k,\theta_k)$ on the extended state space $\Xi=X\times \Theta$, as follows: 

\begin{definition}\label{def_1}
Let $\bar\Lambda_{\text{SLTI}}=(\Xi, \mathcal U, \bar A, \bar D, w, \bar\omega)$ be defined as in \eqref{aug_sys_LTI2} with a safety
specification $\bar\Upsilon=(\Theta_{\textsf{i}}, \Theta_{\textsf{u}}, \bar X_{\textsf{i}},  \bar X_{\textsf{u}}, T)$, and the control inputs $\{u_k=Lx_k\}_{k=0}^{T-1}$. Given $\bar X_{\textsf{i}}\cap\bar X_{\textsf{u}}=\emptyset$, we call a function
$B : \Xi\times \mathcal U \rightarrow R^+_0$ a $\theta$-averaged \textit{control barrier certificate (CBC)} for $\bar\Lambda_{\text{SLTI}}$ if $\eta>0$, $\beta\geq\eta$ and $c>0$ exist such that
\begin{subequations}
  \begin{align}
    &\textsf{E}_{\theta}[B(z)]\leq\eta, \; \forall z\in \bar X_{\textsf{i}}\times \mathcal{U}, \label{eq:a} \\
   & \textsf{E}_{\theta}[B(z)]\geq \beta, \; \forall z\in \bar X_{\textsf{u}}\times \mathcal{U}, \label{eq:b}\\
    &\textsf{E}_{\mathsf{w}}\left[B(z_{k+1})|z_k\right]\leq B(z_k)+c ,\; \forall z\in \Xi\times \mathcal{U}. \label{eq:c}
  \end{align}
\end{subequations}

Due to the stochasticity in $X_{\textsf{i}}(\theta_{\textsf{i}})$ and $X_{\textsf{u}}(\theta_{\textsf{u}})$, the conditions \eqref{eq:a}-\eqref{eq:b} are modified compared to the common CBF definition for deterministic specifications that does not involve $\textsf{E}_{\theta}[\cdot]$. Note that, in the conditions \eqref{eq:a}-\eqref{eq:b}, the expectation applies only with respect to $\theta$, while in \eqref{eq:c} the expectation applies with respect to $\mathsf{w}$.
\end{definition}

\section{Control Barrier Certificate \& Safety Guarantee}
In this section, given the augmented system $\bar\Lambda_{\text{SLTI}}$, the control policy $u_k=Lx_k$, and the safety specification $\bar\Upsilon$, we construct a CBC based on \textit{Definition \ref{def_1}} and under the randomness in the initial and unsafe sets. Before stating the main result of this work, we make the following crucial assumption:
\begin{assumption}[Feasibility]
    There exists a pair $(L,P)$ such that the following bilinear matrix inequality holds
    \begin{equation}
        \bar A^\top P \bar A -P \preceq 0.
    \end{equation}
\end{assumption}

\begin{theorem}\label{thm1}
    Let $\bar\Lambda_{\text{SLTI}}=(\Xi, \mathcal U, \bar A, \bar D, w, \bar\omega)$ be defined as in \eqref{aug_sys_LTI2}, controlled by sequence of inputs $\{u_k=Lx_k\}_{k=0}^{T-1}$, and \textit{Assumption 2} holds. Define $B(z)=z^\top Pz$, where $P\succ 0$ is of appropriate dimension. Then $\eta>0$, $\beta\geq\eta$, and $c>0$ exist such that $B(z)$ is a $\theta$-averaged CBC and $u_k$ is a safety controller for $\bar\Lambda_{\text{SLTI}}$ with respect to $\bar \Upsilon=(\Theta_{\textsf{i}},\Theta_{\textsf{u}},\bar X_{\textsf{i}}, \bar X_{\textsf{u}}, T)$. Moreover, we have
    \begin{subequations}
    \begin{align}\label{eta2}
        &\eta= p_{\emptyset}\lambda_{max}(P_x)\int_{\Theta_{\textsf{i}}}d_{\textsf{i}}^{max}(\theta_{\textsf{i}})^2 T_{\textsf{i}}(\theta_{\textsf{i}})d\theta_{\textsf{i}},\\\label{beta2}
        &\beta= p_{\emptyset}\lambda_{min}(P_x)\int_{\Theta_{\textsf{u}}}d_{\textsf{u}}^{min}(\theta_{\textsf{u}})^2 T_{\textsf{u}}(\theta_{\textsf{u}})d\theta_{\textsf{u}},\\\label{c2}
        & c=\textsf{Tr}\left(\bar D^\top P\bar D\;\textsf{E}\left[\mathsf{w}_k^\top\mathsf{w}_k\right]\right).
        \end{align}
    \end{subequations}  
\end{theorem}

\textit{Proof.} Let $P=\begin{pmatrix} P_x & \mathbf{0}_{n\times 2n} \\ \mathbf{0}_{2n\times n} & P_{\theta}\end{pmatrix}$. From \eqref{eq:c}, we have
\begin{align}\label{cond3}
\textsf{E}\left[B(z_{k+1})|z_k\right]&=\textsf{E}\left[z_{k+1}^\top P z_{k+1}|z_k\right]\\\nonumber
&=\textsf{E}\left[x_k^\top A_{cl}^\top P_x A_{cl}x_{k}|z_k\right]\!+\textsf{E}\left[\mathsf{w}_k^\top\bar D^\top P\bar D  \mathsf{w}_k \right]\\\nonumber
&=x_k^\top A_{cl}^\top P_x A_{cl}x_{k}+\textsf{Tr}\left(\bar D^\top P\bar D\;\textsf{E}\left[\mathsf{w}_k^\top\mathsf{w}_k\right]\right).
\end{align}
From \textit{Assumption 2}, it follows that
    \begin{equation*}
        \bar A^\top P \bar A -P=\begin{pmatrix} A_{cl}^\top & \mathbf{0} \\ \mathbf{0} & \mathbf{0}\end{pmatrix} \begin{pmatrix} P_x & \mathbf{0} \\ \mathbf{0} & P_{\theta}\end{pmatrix}\begin{pmatrix} A_{cl} & \mathbf{0} \\ \mathbf{0} & \mathbf{0}\end{pmatrix}-\begin{pmatrix} P_x & \mathbf{0} \\ \mathbf{0} & P_{\theta}\end{pmatrix}\preceq 0,
    \end{equation*}
which leads to the conditions $A_{cl}^\top P_x A_{cl}-P_x\preceq 0$, and $P_{\theta}\succ 0$ to be satisfied. From \eqref{eq:c} and \eqref{cond3}, we have
\begin{align*}
    x_k^\top (\underbrace{A_{cl}^\top P_x A_{cl}-P_x}_{\preceq0})x_{k}+\textsf{Tr}\left(\bar D^\top P\bar D\;\textsf{E}\left[\mathsf{w}_k^\top\mathsf{w}_k\right]\right)\leq c.
\end{align*}
Condition \eqref{eq:c} is therefore satisfied with
\begin{equation}\label{cond3_value}
    c=\textsf{Tr}\left(\bar D^\top P\bar D\;\textsf{E}\left[\mathsf{w}_k^\top\mathsf{w}_k\right]\right).
\end{equation}
To obtain $\eta$ and $\beta$, recall the definitions of $\bar X_{\textsf{i}}$ and $\bar X_{\textsf{u}}$ in \eqref{X_0}-\eqref{X_u}, and $X_{\textsf{i}}(\theta_{\textsf{i}})$ and $X_{\textsf{u}}(\theta_{\textsf{u}})$ in \eqref{X_0_theta}-\eqref{X_u_theta}. We know the following inequalities hold $\forall y\in \mathbb{R}^n$ and $Q\in \mathbb{R}^{n\times n}$:
\begin{equation}\label{max_min}
   \lambda_{min}(Q)\|y\|^2 \leq y^\top Q y\leq \lambda_{max}(Q)\|y\|^2,
\end{equation}
where $\lambda_{min}(Q)$ and $\lambda_{max}(Q)$ are the smallest and largest eigenvalues of $Q$, respectively. 

From \eqref{X_0}, and noting that under $\bar X_{\textsf{u}}\cap \bar X_{\textsf{i}}=\emptyset$, $\theta_{\textsf{u}}$ does not affect the evolution of $z\in \bar X_{\textsf{i}}$ , we conclude $z\in \bar X_{\textsf{i}}\Leftrightarrow z\in \{(x,\theta_{\textsf{i}})|x\in X_{\textsf{i}}(\theta_{\textsf{i}})\}$. From the definition of $X_{\textsf{i}}(\theta_{\textsf{i}})$ in \eqref{X_0_theta}, it holds that $\theta_{\textsf{i}}\in X_{\textsf{i}}(\theta_{\textsf{i}})$, hence we can write $z\in \bar X_{\textsf{i}}\Leftrightarrow x\in X_{\textsf{i}}(\theta_{\textsf{i}})$. This means the evolution of $z\in \bar X_{\textsf{i}}$ is equivalent to the evolution of $x \in X_{\textsf{i}}(\theta_{\textsf{i}})$, as expected. It follows that if a $B_{\theta_{\textsf{i}}}$ exists such that $B_{\theta_{\textsf{i}}}(x|x\!\in \!X_{\textsf{i}}(\theta_{\textsf{i}}))\leq \bar\eta(\theta_{\textsf{i}})$, then
\begin{equation*}
\textsf{E}_{\theta}[B(z|z\!\in \!\bar X_{\textsf{i}})]\!\leq p_{\emptyset}\int_{\Theta_{\textsf{i}}}\!\bar\eta(\theta_{\textsf{i}})T_{\textsf{i}}(\theta_{\textsf{i}})d\theta_{\textsf{i}}.
\end{equation*}
Let $B_{\theta_{\textsf{i}}}(x)=x^\top P_x x, \forall x\!\in \!X_{\textsf{i}}(\theta_{\textsf{i}})$. From \eqref{max_min} we have
\begin{align}\label{eqival}
    B_{\theta_{\textsf{i}}}(x)\!\leq \lambda_{max}(P_x)\|x\|^2\leq \lambda_{max}(P_x)\!\sup_{x\in X_{\textsf{i}}(\theta_{\textsf{i}})}\!\|x\|^2.
\end{align}
To compute $\sup_{x\in X_{\textsf{i}}(\theta_{\textsf{i}})}\|x\|^2$ for a general convex shape $R$ for $X_{\textsf{i}}^{\text{no}}$, we define the support function $H_R(\vec{v})=\sup_{x\in R}\vec{v}^\top x$, where $\vec{v}$ is the unit vector along the line that connects $\bar c_{\textsf{i}}$ to the origin. The Euclidean norm of $x$ can be expressed as $\|x\|=\sup_{\|\vec{v}\|=1}\vec{v}^\top x$, hence we can write
\begin{align*}
    \sup_{x\in X_{\textsf{i}}(\theta_{\textsf{i}})}\!\|x\|= \sup_{x\in X_{\textsf{i}}(\theta_{\textsf{i}})}\sup_{\|\vec{v}\|=1}\vec{v}^\top x=\sup_{\|\vec{v}\|=1}H_{X_{\textsf{i}}(\theta_{\textsf{i}})}(\vec{v}),
\end{align*}
where from \eqref{X_0_theta}, we get $H_{X_{\textsf{i}}(\theta_{\textsf{i}})}(\vec{v})=\vec{v}^\top\bar c_{\textsf{i}}+(\bar s_{\textsf{i}}+\theta_{\textsf{i}})H_R(\vec{v})$. We then obtain
\begin{equation}\label{eq:sup_x_0}
    \sup_{x\in X_{\textsf{i}}(\theta_{\textsf{i}})}\!\|x\|=\!\sup_{\|\vec{v}\|=1}[\vec{v}^\top \bar c_{\textsf{i}}+(\bar s_{\textsf{i}}+\theta_{\textsf{i}})H_R(\vec{v})]\triangleq d_{\textsf{i}}^{max}(\theta_{\textsf{i}}).
\end{equation}
From \eqref{eqival} and \eqref{eq:sup_x_0}, we find $\forall x\in X_{\textsf{i}}(\theta_{\textsf{i}})$ that
\begin{equation*}
    B_{\theta_{\textsf{i}}}(x)\leq \lambda_{max}(P_x)d_{\textsf{i}}^{max}(\theta_{\textsf{i}})^2\triangleq \bar\eta(\theta_{\textsf{i}}).
\end{equation*}
Since $p_{\emptyset}=\textsf{P}[ X_{\textsf{i}}(\theta_{\textsf{i}})\cap X_{\textsf{u}}(\theta_{\textsf{u},0})=\emptyset]=1-p^{ol}_0$, we have
\begin{align*}
\textsf{E}_{\theta}&\left[B_{\theta_{\textsf{i}}}(x)\right]=p_{\emptyset}\textsf{E}_{\theta_{\textsf{i}}}\left[B_{\theta_{\textsf{i}}}(x)| X_{\textsf{i}}(\theta_{\textsf{i}})\cap X_{\textsf{u}}(\theta_{\textsf{u},0})=\emptyset\right]\\&\leq p_{\emptyset}\lambda_{max}(P_x)\textsf{E}_{\theta_{\textsf{i}}}\!\left[d_{\textsf{i}}^{max}(\theta_{\textsf{i}})^2|X_{\textsf{i}}(\theta_{\textsf{i}})\cap X_{\textsf{u}}(\theta_{\textsf{u},0})=\emptyset\right]\\
&=p_{\emptyset}\lambda_{max}(P_x)\int_{\Theta_{\textsf{i}}}d_{\textsf{i}}^{max}(\theta_{\textsf{i}})^2 T_{\textsf{i}}(\theta_{\textsf{i}})d\theta_{\textsf{i}}=\eta.
\end{align*}
The integral can be readily computed for any convex shape kernel $R$ and transition kernel $T_{\textsf{i}}$, hence \eqref{eta2} is confirmed.

Since $p_{\theta}\succ 0$, we have $z^\top P z\geq x^\top P_x x$, and we reach to
\begin{align*}
    \|z\|_{P}^2&\geq \|x\|^2_{P_x}\geq\lambda_{min}(P_x)\|x\|^2\geq \lambda_{min}(P_x) \inf_{x}\|x\|^2.
\end{align*}
For all $z\in \bar X_{\textsf{u}}$, we have
\begin{align*}
    B(z)\geq x^\top \!P_x x\geq \lambda_{min}(P_x)\inf_{z\in \bar X_{\textsf{u}}}\|x\|^2.
\end{align*}
With similar reasoning this time for $z\in \bar X_{\textsf{u}}$, the equivalence $z\in \bar X_{\textsf{u}} \Leftrightarrow x\in X_{\textsf{u}}(\theta_{\textsf{u}})$ holds, and the infimum above can then be taken over $x\in X_{\textsf{u}}(\theta_{\textsf{u}})$. To compute the minimum distance between the origin and any point $x\in X_{\textsf{u}}(\theta_{\textsf{u}})$, recall that $\theta_{\textsf{u}}$ is a sequence of i.i.d. random variables, hence, the uncertainty in $X_{\textsf{u}}(\theta_{\textsf{u}})$ evolves with time. We have from \eqref{X_u_theta} that
\begin{equation*}
    \inf_{x\in X_{\textsf{u}}(\theta_{\textsf{u}})}\!\!\|x\|\!=\!\!\sup_{\|\vec{v}'\|=1}\![\vec{v}'^\top \bar c_{\textsf{u}}+(\bar s_{\textsf{u}}+\theta_{\textsf{u}})H_R(-\vec{v}')]\triangleq d_{\textsf{u}}^{min}(\theta_{\textsf{u}}),
\end{equation*}
with $\vec{v}'$ the unit vector along the line connecting $\bar c_{\textsf{u}}$ to the origin\footnote{If we assume the initial set is origin-centered, then $\vec{v}=-\vec{v}'$.}. Hence, $B(z)\geq \lambda_{min}(P_x)d_{\textsf{u}}^{min}(\theta_{\textsf{u}})^2$, and we have
\begin{align*}
\textsf{E}_{\theta}\!&\left[B(z)\right]\geq\\& p_{\emptyset}\lambda_{min}(P_x)\;\textsf{E}_{\theta_{\textsf{u}}}\!\left[d_{\textsf{u}}^{min}(\theta_{\textsf{u}})^2| X_{\textsf{i}}(\theta_{\textsf{i}})\cap X_{\textsf{u}}(\theta_{\textsf{u}})=\emptyset\right].
\end{align*}
Since $\theta_{\textsf{u}}$ evolves with time, we essentially need to take $T$ times integral to compute $\textsf{E}_{\theta_{\textsf{u}}}\!\!\left[d_{\textsf{u}}^{min}(\theta_{\textsf{u}})^2\right]$, however, $\{\theta_{\textsf{u},k}\}_{k}$ are i.i.d., hence they have identical probability transition kernels, and we readily find $\beta$ as
\begin{align*}
\textsf{E}_{\theta}\left[B(z)\right]&\geq p_{\emptyset}\lambda_{min}(P_x)\int_{\Theta_{\textsf{u}}}d_{\textsf{u}}^{min}(\theta_{\textsf{u}})^2 T_{\textsf{u}}(\theta_{\textsf{u}})d\theta_{\textsf{u}}=\beta.
\end{align*}
Since $\beta\leq \mathsf{E}[B(z)]\leq\eta$, $\beta\geq \eta$ holds, and $B(z)=z^\top Pz$ is a CBC for $\bar\Lambda_{\text{SLTI}}$ with respect to $\bar \Upsilon$, and under the control inputs $\{u_k=Lx_k\}_{k=0}^{T-1}$, and the proof is complete.
\qed

\textit{Remark 3.} As observed in \textit{Theorem \ref{thm1}}, conditions \eqref{eta2} and \eqref{beta2} depend only on $P_x$ while condition \eqref{c2} depends on $P_x$ and $P_{\theta}$. This property is the result of transforming the stochastic sets $X_{\textsf{i}}(\theta_{\textsf{i}})$ and $X_{\textsf{u}}(\theta_{\textsf{u}})$ in \eqref{X_0_theta}-\eqref{X_u_theta} to deterministic sets $\bar X_{\textsf{i}}$ and $\bar X_{\textsf{u}}$ in \eqref{X_0}-\eqref{X_u} on the extended space.

\begin{corollary}[Ball-shaped sets] Let the shape kernel $R$ for the initial and unsafe sets be defined as the unit $n$-dimensional ball centered, respectively, at $\bar c_{\textsf{i}}$ and $\bar c_{\textsf{u}}$, and with nominal radii $r_{\textsf{i}}$ and $r_{\textsf{u}}$. Moreover, assume $\theta_{\textsf{i}}\sim |\mathcal N(0,\sigma_{\textsf{i}}^2)|$ and $\theta_{\textsf{u},k}\sim |\mathcal N(0,\sigma_{\textsf{u}}^2)|, \forall k\in[0,T]$, and we have $X_{\textsf{i}}(\theta_{\textsf{i}})=\bar c_{\textsf{i}}\oplus(r_{\textsf{i}}+\theta_{\textsf{i}})$, and $X_{\textsf{u}}(\theta_{\textsf{u},k})=\bar c_{\textsf{u}}\oplus(r_{\textsf{u}}+\theta_{\textsf{u},k})$. Then the CBC conditions in \textit{Theorem \ref{thm1}} reduce to
    \begin{subequations}
    \begin{align}\label{eta3}
        &\eta_b= p_{\emptyset}\lambda_{max}(P_x)\left[\gamma_{\textsf{i}}^2+2\gamma_{\textsf{i}}\sigma_{\textsf{i}}\sqrt{\frac{2}{\pi}}+\sigma_{\textsf{i}}^2\right],\\\label{beta3}
        &\beta_b= p_{\emptyset}\lambda_{min}(P_x)\left[\gamma_{\textsf{u}}^2-2\gamma_{\textsf{u}}\sigma_{\textsf{u}}\sqrt{\frac{2}{\pi}}+\sigma_{\textsf{u}}^2)\right],\\\label{c3}
        & c_b=\textsf{Tr}(P_x I_{n}\sigma_w^2)+ \textsf{Tr}(P_{\theta}\Sigma_{\theta}),
        \end{align}
    \end{subequations}
    with $\gamma_{\textsf{i}}=\|\bar c_{\textsf{i}}\|+r_{\textsf{i}}$, $\gamma_{\textsf{u}}=\|\bar c_{\textsf{u}}\|-r_{\textsf{u}}$, and $\Sigma_{\theta}\!=\!\begin{pmatrix} \!I_{n}\sigma_{\textsf{i}}^2&\!\mathbf{0}\\\!\mathbf{0}&\!I_{n}\sigma_{\textsf{u}}^2\end{pmatrix}$.
\end{corollary}

Next, we state the formal safety probability using the CBC construction results in \textit{Theorem \ref{thm1}}. The proof mainly follows similar steps as in \cite{1428804}, however, we modify the supermartingale according to our modified CBC conditions in \eqref{eq:a}-\eqref{eq:c}, and set uncertainties.

\begin{theorem}\label{thm2}
     Let $\bar\Lambda_{\text{SLTI}}=(\Xi, \mathcal U, \bar A, \bar D, w, \bar\omega)$ be defined as in \eqref{aug_sys_LTI2}, controlled by sequence of inputs $\{u_k=Lx_k\}_{k=0}^{T-1}$. Let a CBC $B(z)$ exists that satisfies the conditions \eqref{eq:a}-\eqref{eq:c}. The probability that any closed-loop trajectory $\{z_k\}_{k=0}^T$ of $\bar\Lambda_{\text{SLTI}}$ from any initial state $z_0\in \bar X_{\textsf{i}}$ does not reach to the set $\bar X_{\textsf{u}}$ over time horizon $[0,T]$ is
     \begin{equation}\label{safe_prob_final}
    p_{safe=}\mathsf{P}_{w,\theta_{\textsf{i}},\theta_{\textsf{u}}}\{\bar\Lambda_{\text{SLTI}}\models \bar\Upsilon\}\geq 1-\frac{\eta+cT}{\beta},
\end{equation}
with $\eta$, $\beta$, and $c$ defined in \eqref{eta2}-\eqref{c2} in \textit{Theorem \ref{thm1}}.
\end{theorem}

\textit{Proof.}
Define $W(z_k)=B(z_k)+c(T-k)$. From \textit{Definition 1} and for a CBC $B(z_k)$, we have $\mathsf{E}_{\mathsf{w}}[B(z_{k+1})|z_k]\leq B(z_k)+c, \forall z_k$. We then conclude
\begin{align*}
    \textsf{E}_{\mathsf{w}}[W(z_{k+1})|z_k]&=\textsf{E}_{\mathsf{w}}[B(z_{k+1})|z_k]+c(T-k-1)\\
    &\leq B(z_{k})+c+c(T-k-1)
    \\&=B(z_{k})+c(T-k)=W(z_k),
\end{align*}
Since $\textsf{E}_{\mathsf{w}}[W(z_{k+1})|z_k]\leq W(z_k)$ holds for all $k$, the process $\{W_k\}_{k=0}^T$ is a non-negative supermartingale. Moreover, from condition \eqref{eq:a} we have $\textsf{E}_{\theta}[B(z_0)]\leq \eta, \forall z_0\in \bar X_{\textsf{i}}$, hence, we obtain $\textsf{E}_{\theta}[W(z_0)]=\textsf{E}_{\theta}[B(z_0)]+cT\leq \eta+cT$. 

Define hitting state $h_s\triangleq\{\exists k\in [0,T]:z_k\in \bar X_{\textsf{u}}\}$, then the probability that $\bar\Lambda_{\text{SLTI}}$ becomes unsafe is $p_{unsafe}=\textsf{P}_{\mathsf{w}}[h_s|z_0\in \bar X_{\textsf{i}}]$. Given $z_0\in \bar X_{\textsf{i}}$, we define the  random process $\mathbf{1}_{\bar X_{\textsf{u}}}(z_k)=\begin{cases}
    1, &z_k\in \bar X_{\textsf{u}}\\
    0, & z_k\notin \bar X_{\textsf{u}}
\end{cases}$, which returns only two outcomes. We can write
\begin{equation*}
    p_{unsafe}=\textsf{E}\left[\mathbf{1}_{\bar X_{\textsf{u}}}(z_k)\right]=1\times \textsf{P}[z_k\in \bar X_{\textsf{u}}]+ 0\times \textsf{P}[z_k\notin \bar X_{\textsf{u}}].
\end{equation*}
From condition \eqref{eq:b}, $z_k\in \bar X_{\textsf{u}}$ is equivalent with $\mathsf{E}_{\theta}[B(z_k)]\geq \beta$. Therefore, the following holds
\begin{align*}
    p_{unsafe}=\textsf{P}[z_k\in \bar X_{\textsf{u}}]=\textsf{P}\left[\mathsf{E}_{\theta}[B(z_k)]\geq \beta\right]\leq \frac{\mathsf{E}\left[\mathsf{E}_{\theta}[B(z_k)]\right]}{\beta},
\end{align*}
where the right-hand side expression is obtained by applying Markov's inequality. From the definition of $W(z_k)$, we know $\textsf{E}[W(z_k)]\geq \textsf{E}[B(z_k)]$, hence it follows that
\begin{equation*}
        p_{unsafe}\leq \frac{\mathsf{E}\left[\mathsf{E}_{\theta}[B(z_k)]\right]}{\beta}\leq \frac{\mathsf{E}\left[\mathsf{E}_{\theta}[W(z_k)]\right]}{\beta}.
\end{equation*}
Since $\{W_k\}_{k=0}^T$ is a non-negative supermartingale, the monotonicity property holds as $\textsf{E}[W(z_k)]\leq \textsf{E}[W(z_{k-1})]\leq \ldots \leq\textsf{E}[W(z_{0})]$, hence $\textsf{E}[W(z_k)]\leq \textsf{E}[W(z_0)]$. Then we have 
\begin{equation*}
        p_{unsafe}\leq \frac{\mathsf{E}\left[\mathsf{E}_{\theta}[W(z_k)]\right]}{\beta}\leq \frac{\mathsf{E}\left[\mathsf{E}_{\theta}[W(z_0)]\right]}{\beta}\leq\frac{\eta+cT}{\beta}.
\end{equation*}
The probability of never hitting the unsafe set is the complementary probability of $p_{unsafe}$, hence,  $\mathsf{P}_{w,\theta_{\textsf{i}},\theta_{\textsf{u}}}\{\bar\Lambda_{\text{SLTI}}\models \bar\Upsilon\}=p_{safe}\geq 1-\frac{\eta+cT}{\beta}$, and proof is complete. \qed

\section{Simulation Results}
\label{sec:simulations}

We consider the model of an RLC circuit as a discrete-time stochastic system with the following dynamics (\cite{akbarzadeh2025})
\begin{equation}\label{sym_model}
			\begin{pmatrix} x^1_{k+1}\\x^2_{k+1}\end{pmatrix}\!=\!\begin{pmatrix} 1-\frac{\Delta R}{L} &\;\;-\frac{\Delta}{L}\\ \!\!\!\frac{\Delta}{C}&\;\;\;\;\;1\end{pmatrix}\!\begin{pmatrix} x^1_{k}\\x^2_{k}\end{pmatrix}\!+\!\begin{pmatrix} u^1_{k}\\u^2_{k}\end{pmatrix}\!+\!\begin{pmatrix} w^1_{k}\\w^2_{k}\end{pmatrix},
\end{equation}
where $x^1_{k}$ and $x^2_{k}$ represent the current and voltage of the circuit, respectively, $\Delta = 0.05$ is the sampling time, the series resistance, inductance, and capacitance are, respectively, $R = 2$, $L = 9$, and $C = 0.5$, and $\{w^1_{k},w^2_{k}\}_{k=0}^T \sim \mathcal{N}(0,\sigma_w^2)$. The state space of the system is $X = [-4,10]\times [-4,10]$, the nominal initial set is $X_{\textsf{i}}^{\text{no}} = \textsf{Ball}((0,0),0.4)$, and the nominal unsafe set is $X_{\textsf{u}}^{\text{no}}= \textsf{Ball}((4,4),1)$. The initial and unsafe sets are associated with random uncertainties $\theta_{\textsf{i}}\sim|\mathcal N(0,\sigma_{\textsf{i}}^2)|$ and $\{\theta_{\textsf{u},k}\}_{k=0}^T\sim|\mathcal N(0,\sigma_{\textsf{u}}^2)|$, respectively. The aim is to construct a CBF and the corresponding safety controller ensuring that the system states starting from $X_{\textsf{i}}(\theta_{\textsf{i}})$ remain within the safe set $X \backslash X_{\textsf{u}}(\theta_{\textsf{u}})$ over the time horizon $[0,T]$, where we set the time duration of system evolution to $T=50$ time-steps.

First, we compute a pair $(L, P)$ that satisfies \textit{Assumption~2} for the given system dynamics in \eqref{sym_model} with the augmented state $z_k=[x_k^1,x_k^2,\theta_{\textsf{i}}, \theta_{\textsf{u},k}]^\top$, as follows:
\begin{align*}
    L&=\begin{pmatrix}
        -0.0337 &  -0.0400\\-0.0401  & -0.0476
    \end{pmatrix}, \; P_x=\begin{pmatrix}
         0.0133   &     0\\0  &  0.0120
    \end{pmatrix},\\ P_{\theta}&=\begin{pmatrix}
         10^{-6}   &     0\\0  &  1.57\!\times\!10^{-4} \;\sigma_{\textsf{i}}^{(\frac{1}{\sigma_{\textsf{i}}}-1)/\sigma_{\textsf{i}}^{1.9}}\!
    \end{pmatrix}.
\end{align*}
It is straightforward to check that above parameters satisfy $A_{cl}^\top P_xA_{cl}-P_x\preceq 0$ and $P_{\theta}\succ 0, \forall \sigma_{\textsf{i}}>0$, as in \textit{Theorem \ref{thm1}}.

The uncertain initial and unsafe sets are $X_{\textsf{i}}(\theta_{\textsf{i}})=\textsf{Ball}((0,0),0.4+\theta_{\textsf{i}})$, and $X_{\textsf{u}}(\theta_{\textsf{u},k})=\textsf{Ball}((4,4),1+\theta_{\textsf{u},k})$, and we consider only positive random variables $\theta_{\textsf{i}}$ and $\theta_{\textsf{u},k}, \forall k$, to account for the non-trivial scenario wherein uncertainties lead to inflating the initial and unsafe sets.

\begin{figure}[b!]\vspace{3mm}
    \centering
    \includegraphics[width=.9\linewidth]{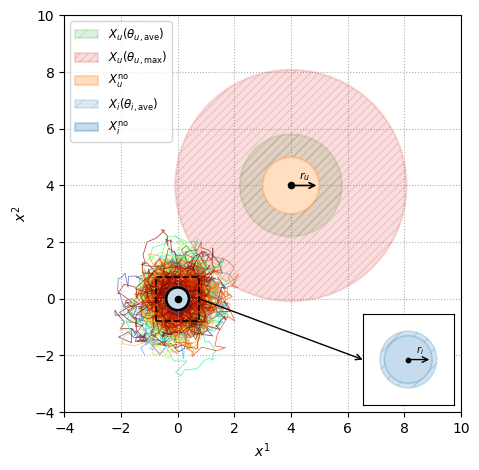}\vspace{-3mm}
    \caption{Monte Carlo simulation of system trajectories and set variations with 200 trajectory samples, and $\sigma_w=0.2$, $\sigma_{\textsf{i}}=0.1$, $\sigma_{\textsf{u}}=1$. Safety probability over 50 time-steps is computed as $99.90\%$, while the analytical lower-bound on safety probability is $70.66\%$.}
    \label{fig:sample_traj}
\end{figure}

Fig.~\ref{fig:sample_traj} demonstrates Monte Carlo simulation outcomes for a sample scenario with $\sigma_w=0.2$, $\sigma_{\textsf{i}}=0.1$, and $\sigma_{\textsf{u}}=1$, where 200 trajectories starting from the set $X_{\textsf{i}}(\theta_{\textsf{i}})$ evolve for $50$ time-steps, where the desired control objective is to maintain trajectories of system \eqref{sym_model} within the nominal initial set. For each of the 200 sample trajectories, we draw a $\theta_{\textsf{i}}$ and the initial states $x_0^1$ and $x_0^2$ are selected by a uniform distribution over the set $X_{\textsf{i}}(\theta_{\textsf{i}})$. The unsafe set $X_{\textsf{u}}(\theta_{\textsf{u},k})$ is determined at each time $k$ by the realization of the random variable $\theta_{\textsf{u},k}$. We observe that the nominal initial (solid blue circle) and unsafe (solid orange circle) sets are enlarged by the uncertainties, which affect system's safety, i.e., the probability that trajectories hit the unsafe set. The pink cross-hatched circle shows the maximum unsafe set, and the green cross-hatched circle is the unsafe set averaged over 200 samples. Moreover, the blue cross-hatched circle demonstrates the average initial set over all 200 samples. The analytical safety probability for this specific scenario, as in \textit{Theorem~\ref{thm2}}, is $p_{\text{safe}}>70.66\%$, with $\eta=0.003109$, $\beta=0.183054$, and $c=0.001012$.

Fig.~\ref{fig:safprob} depicts comparisons between safety probabilities obtained from simulations (each averaged over 20000 samples with normal $95\%$ confidence interval) and the analytical lower-bound for safety probabilities, for a range of system and initial/unsafe set uncertainties with $\sigma_w=\{0.01, 0.05, 0.1,0.15,0.2\}, \;\sigma_{\textsf{i}}=\{0.5,0.75,1.0,1.5, 1.75\}$, and $\sigma_{\textsf{u}}\!=\!\{0.1, 0.5, 0.75, 1.0, 1.5, 1.75\}$. In Figs.~\eqref{fig:safprob_sigma0_0_50}-\eqref{fig:safprob_sigma0_1_75}, we observe that our obtained analytical safety probabilities are valid lower-bounds for the simulative probabilities for each trio of $\{\sigma_w, \sigma_{\textsf{i}}, \sigma_{\textsf{u}}\}$. 

Safety probabilities for the system in \eqref{sym_model} with the nominal safety specifications (i.e., $\sigma_{\textsf{i}}=\sigma_{\textsf{u}}=0$), and in the presence of system uncertainty $(w_k^1,w_k^2)$ are $\{0.991, 0.979, 0.943,0.882,0.7973\}$, respectively for $\sigma_w=\{0.01, 0.05, 0.1,0.15,0.2\}$. Compared to safety probabilities in Figs.~\eqref{fig:safprob_sigma0_0_50}-\eqref{fig:safprob_sigma0_1_75}, we observe that uncertainty in specifications notably change safety properties, hence any formal result on safety should account for such uncertainties.

As it can be seen in Figs.~\eqref{fig:safprob_sigma0_0_50}-\eqref{fig:safprob_sigma0_1_75}, the analytical lower-bounds become more conservative as $\sigma_w$ and $\sigma_{\textsf{i}}$ increase, where the conservativeness with respect to $\sigma_w$ is more notable. The main reason is that we consider a quadratic CBF that considers worst-cases system uncertainties $w_k^1, w_k^2$ (always positive), while in the simulations they can take negative values as well. 

\textit{Remark 4.} An interesting result of this work is that extending the state space, as discussed in Section~\ref{sec:ext_space}, allows us to flexibly design $P_{\theta}\succ 0$. Since the lower-bound for the safety probability in \textit{Theorem \ref{thm2}} depends on $P_{\theta}$ due to extending the state space, we can design $P_{\theta}$ such that it reduces the conservativeness of the lower-bounds. Clearly, no unique $P_{\theta}$ could manage to reduce conservativeness for all ranges of uncertainties, however, there is no constraint on $P_{\theta}$ to be static as long as it remains positive definite.
\begin{figure}[t]
    \centering
    
    \begin{subfigure}{0.24\textwidth}
        \centering
        \includegraphics[width=\linewidth]{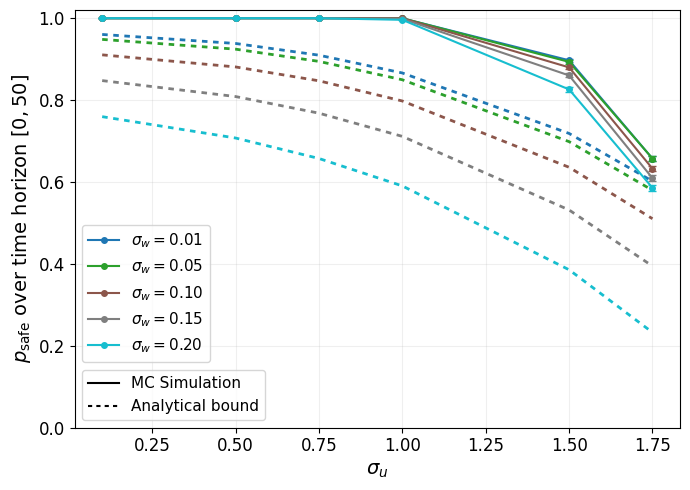}\vspace{-1mm}
        \caption{$\sigma_{\textsf{i}} = 0.5$}
        \label{fig:safprob_sigma0_0_50}
    \end{subfigure}
    \hfill
    \begin{subfigure}{0.24\textwidth}
        \centering
        \includegraphics[width=\linewidth]{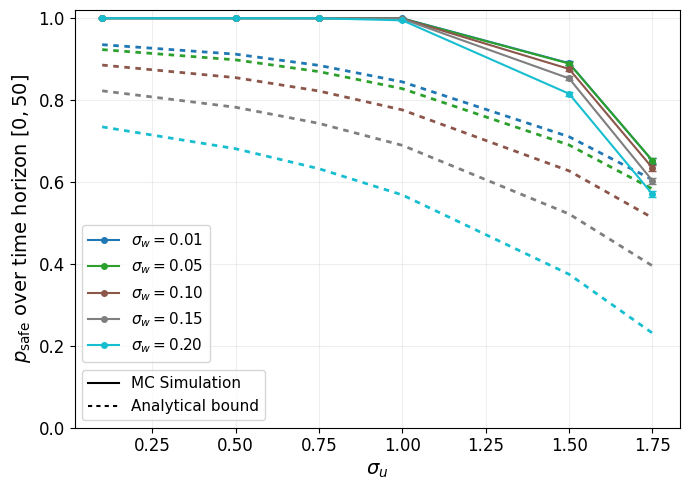}\vspace{-1mm}
        \caption{$\sigma_{\textsf{i}} = 0.75$}
        \label{fig:safprob_sigma0_0_75}
    \end{subfigure}
    
    \vspace{0.4cm}
    
    \begin{subfigure}{0.24\textwidth}
        \centering
        \includegraphics[width=\linewidth]{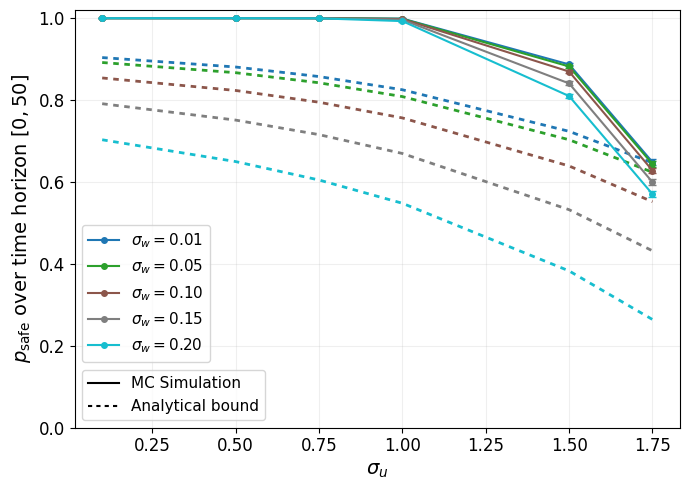}\vspace{-1mm}
        \caption{$\sigma_{\textsf{i}} = 1.0$}
        \label{fig:safprob_sigma0_1_00}
    \end{subfigure}
    \hfill
    \begin{subfigure}{0.24\textwidth}
        \centering
        \includegraphics[width=\linewidth]{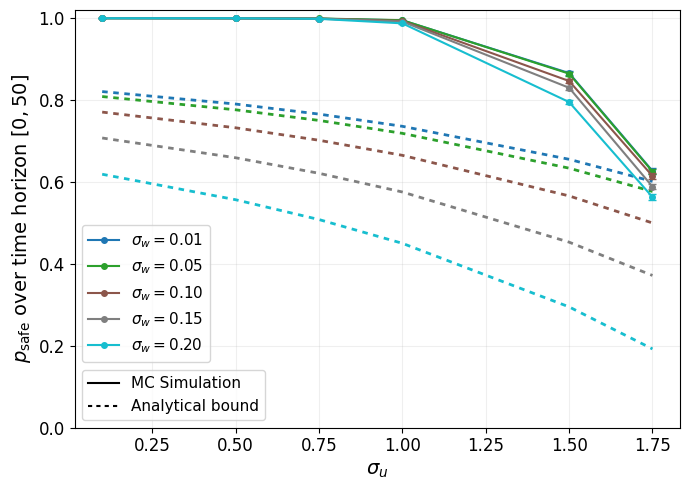}\vspace{-1mm}
        \caption{$\sigma_{\textsf{i}} = 1.5$}
        \label{fig:safprob_sigma0_1_50}
    \end{subfigure}
    
    \vspace{0.4cm}
    
    \begin{subfigure}{0.25\textwidth}
        \centering
        \includegraphics[width=\linewidth]{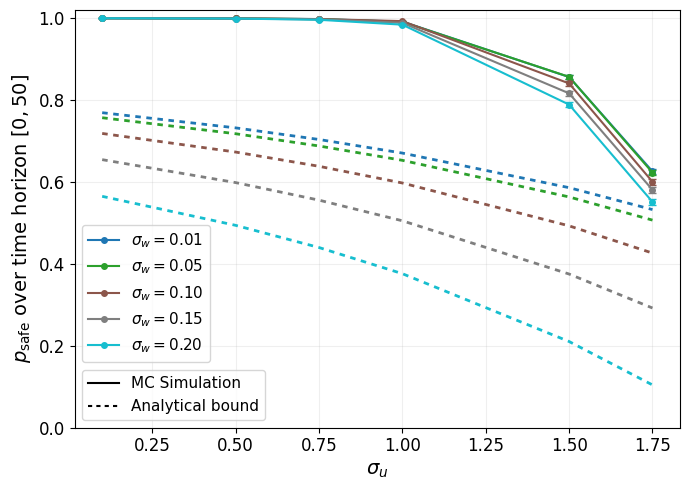}\vspace{-1mm}
        \caption{$\sigma_{\textsf{i}} = 1.75$}
        \label{fig:safprob_sigma0_1_75}
    \end{subfigure}
    \caption{Analytical (dashed) vs. Monte Carlo (solid) safety probabilities for various ranges of $\{\sigma_w, \sigma_{\textsf{i}}, \sigma_{\textsf{u}}\}$.}
    \label{fig:safprob}
\end{figure}

\section{Conclusion}
In this paper, we study the application of barrier certificates to satisfy temporal logic specifications with randomly evolving predicates on stochastic dynamical systems. By introducing an augmented system that includes the states of the dynamical system and the random processes in the predicates, evolving on the extended product space,, we reformulate the satisfaction of uncertain temporal specifications in the system's state space as the satisfaction of deterministic specifications on the extended space. We then apply barrier certificates on the specific class of linear time-invariant stochastic systems to analytically derive probabilistic safety bounds on the system's trajectories. We validate our analytical safety bounds by comparing them with simulations for a wide range of stochasticity in the predicates of the temporal specifications.



\section*{DECLARATION OF AI-ASSISTED TECHNOLOGIES IN THE WRITING PROCESS}
During the preparation of this work, the authors exploited GPT-5 by OpenAI in order to improve
writing style. After using this tool, the authors
reviewed and edited the content as needed and take full responsibility for the content of the publication.

\bibliography{ifacconf}             
                                                   







\end{document}